\documentclass[a4paper,11pt,twoside,notitlepage,reqno]{amsart}

\usepackage{amsmath,amssymb,amsbsy,amsthm}
\usepackage[hmargin=1.4in,vmargin=1.4in]{geometry}
\usepackage{url}

\usepackage{tikz}
\usepackage{tikz-cd}    
\usepackage{amsmath}    
\usepackage{amssymb}    
\usepackage{amsthm}     
\usetikzlibrary{matrix}
\usepackage[linktocpage]{hyperref}
\hypersetup{
    colorlinks=true,        
    linkcolor=red,          
    citecolor=red,         
    filecolor=magenta,      
    urlcolor=cyan           
}

\addtocontents{toc}{\protect\setcounter{tocdepth}{1}}

\usepackage{eucal}

\usepackage{algorithm}
\usepackage{algpseudocode}

\theoremstyle{plain}
	\newtheorem{thm}{Theorem}

		\numberwithin{thm}{section}
	\newtheorem{lemma}[thm]{Lemma}
	\newtheorem{prop}[thm]{Proposition}

	\newtheorem*{thm*}{Theorem}
	\newtheorem*{lemma*}{Lemma}
	\newtheorem*{prop*}{Proposition}
	\newtheorem*{cor*}{Corollary}
	\newtheorem*{conj*}{Conjecture}

\theoremstyle{definition}
	
	\newtheorem*{example*}{Example}
	
	\newtheorem{remark}[thm]{Remark}

\begin{document}

\title[]
{Filtered deformations of three-variable polynomial algebras}

\author{Jason Bell}
\author{Boris Li}
\address{University of Waterloo \\
Department of Pure Mathematics \\
Waterloo, Ontario \\
Canada  N2L 3G1}
\thanks{The authors were supported in part by NSERC grant RGPIN-2022-02951.}
\email{jpbell@uwaterloo.ca}
\email{boris.li@uwaterloo.ca}
\begin{abstract}
    We classify all filtered deformations of polynomial rings in three variables.  We use our characterization to answer a special case of a question of Etingof by showing that if $F$ is an algebraically closed field of positive characteristic and $A$ is an $F$-algebra with a filtration by finite-dimensional subspaces whose associated graded ring is a polynomial ring in three variables then $A$ satisfies a polynomial identity.  \end{abstract}
\maketitle

\section{Introduction}
In this paper, we look at ``quantized'' algebras that are constructed by deforming the multiplicative structure of a commutative polynomial ring. A general way of building such quantized algebras comes from studying algebras that are filtered and whose associated graded algebra with respect to this filtration is exactly this polynomial ring.  Given a field $K$ and a graded $K$-algebra $B$, a $\mathbb{Z}_+$-\emph{filtered deformation} (henceforth we shall drop the $\mathbb{Z}_+$) of $B$ is simply an algebra $A$ with a filtration 
\begin{equation}\label{eq:filt}
V_0\subseteq V_1 \subseteq V_2 \subseteq \cdots \subseteq \bigcup_i V_i=A
\end{equation} such that $V_iV_j\subseteq V_{i+j}$ for $i,j\ge 0$ and such that $B\cong \bigoplus_{n\ge 0} V_n/V_{n-1}$, where we take $V_{-1}=(0)$; i.e., $B$ is the associated graded algebra of $A$.  An important class of filtered deformations of polynomial rings comes from enveloping algebras over a base field $K$ where the filtration is given by $V_n=V^n$, and where $V$ is the subspace spanned by the image of the Lie algebra and $1$.  In positive characteristic, enveloping algebras of finite-dimensional Lie algebras are known to be well-behaved.  In particular, they satisfy a polynomial identity.  One can think of this as saying that filtered deformations of polynomial algebras that are generated in degree one behave much like algebras that are finite over their centres.  Etingof \cite[Question 1.1]{Etingof} asked whether this phenomenon holds more generally for all filtered deformations of graded commutative algebras over positive characteristic algebraically closed fields and he proved that when a filtered deformation of a commutative algebra
in positive characteristic satisfies a polynomial identity, then its
PI-degree is a power of the characteristic of the field.

Understanding when filtered deformations in positive characteristic satisfy a polynomial identity is often useful when looking at filtered deformations of finitely generated commutative algebras in characteristic zero. The reason for this is that the algebras are finitely presented and hence have models over a base ring that is finitely generated as a $\mathbb{Z}$-algebra.  One can then reduce mod maximal ideals of the base ring to obtain filtered deformations of algebras in positive characteristic, and the fact that these have large centres in positive characteristic can sometimes be used to understand subtle phenomena in the characteristic zero setting. In some circumstances, for instance, this approach can be used to lift mod $p$ information to prove non-trivial results in characteristic zero; this philosophy inspired a related question of Cuadra, Etingof, and Walton \cite{CEW}.   

In general, Etingof's question has proved difficult to answer, although it has been answered in the affirmative in the case of iterated Hopf Ore extensions in positive characteristic \cite{BZ} and for algebras of Gelfand-Kirillov dimension at most two \cite{Bell}.  Etingof \cite{Etingof} notes that the case of filtered deformations of polynomial rings is already non-trivial and is of special interest. In this paper, we prove that Etingof's question has an affirmative answer in the case when the associated graded algebra is a polynomial ring in three variables.  
\begin{thm}
Let $K$ be an algebraically closed field of positive characteristic and let $A$ be a $K$-algebra with an exhaustive filtration by finite-dimensional $K$-vector subspaces, whose associated graded ring is a polynomial algebra over $K$ in three variables.  Then $A$ satisfies a polynomial identity.
\label{thm:main2}
\end{thm}

Our approach to proving Theorem \ref{thm:main2} is admittedly ham-fisted and relies on first classifying filtered deformations of three-variable polynomial rings, although we believe this classification is interesting in its own right.  Even armed with a classification, showing the algebras satisfy polynomial identities in positive characteristic is non-trivial for several of the families we obtain.  The reason for this is that in many cases, the families have models over characteristic zero fields that don't satisfy a polynomial identity and so the PI degrees of the algebras grow with the characteristic of the field; moreover, in many cases there are not obvious central elements one can work with. 

An important and natural class of filtered deformations of polynomial rings comes from \emph{iterated Ore extensions of unipotent type}.  We recall that a skew polynomial ring over a base $K$-algebra $R$ is $R[x]$ as a set but with multiplication ``deformed'' by extending the multiplication on $R$ to $R[x]$ via 
$x\cdot r = \sigma(r) x +\delta(r)$ where $\sigma$ is a $K$-linear automorphism of $R$ and $\delta$ is a $\sigma$-derivation; i.e., $\delta$ is $K$-linear and $\delta(ab)=\delta(a)b+\sigma(a)\delta(b)$ for $a,b\in R$.  We write $R[x;\sigma,\delta]$ to denote the resulting algebra. We say that an automorphism $\sigma$ of $R$ is \emph{unipotent} if the operator $\Delta = \sigma-{\rm Id}$ is locally nilpotent. If $R$ is finitely generated and $K$ has characteristic $p>0$ then this says $\sigma$ has order a power of $p$, since $\sigma^{p^n} = {\rm Id} +\Delta^{p^n}$ and for $n$ large enough, $\Delta^{p^n}$ annihilates a generating set of $R$.  We then say that a ring is an iterated Ore extension of $K$ if it is isomorphic to an algebra of the form $K[x_1][x_2;\sigma_2,\delta_2]\cdots [x_n;\sigma_n,\delta_n]$ and if, moreover, each $\sigma_i$ is unipotent, we say that it is of unipotent type.

A careful study yields the following lengthy classification.  We note that many of the families only exist in characteristics $2$ or $3$.
\begin{thm}\label{thm:main}
Let $K$ be a field and let $A$ be a filtered deformation of a three-variable polynomial ring over $K$.  Then $A$ is isomorphic to one of the following algebras:
\begin{enumerate}
\item A filtration-compatible iterated Ore extension
\[
K[x_1][x_2;\sigma_2,\delta_2][x_3;\sigma_3,\delta_3]
\]
of $K$ of unipotent type, where $x_1,x_2,x_3$ are the generators $x,y,z$ in a
suitable order;
\item The algebra with generators $x,y,z$ and relations $$[x,y]=z, \qquad [x,z]=\gamma y, \qquad [y,z]=f(x)$$ for some $\gamma\in K$ and some $f(x)\in K[x]$.

\item The algebra with generators $x,y,z$ and relations
$$[x,y]=z, \quad [x,z]=\delta\beta(z-2xy)+\delta y^2+\beta x^3+\gamma_2 x^2+\gamma_1 x+\gamma_0,$$ 
$$[y,z]=u(x)-\Big(\gamma_1+(2\gamma_2-2\delta\beta^2)x+3\beta x^2\Big)y+\Big(\gamma_2-\delta\beta^2+3\beta x\Big)z.$$
 $\delta, \beta,\gamma_0,\gamma_1,\gamma_2\in K$, $u(x)\in K[x]$ of degree at most four; moreover, we can take either $\delta=0$; $\delta\neq 0, \beta=\delta^{-1}$; $\delta\neq 0, \beta=0$.
 
\item The algebra with generators $x,y,z$ and relations
$$[x,y]=z, [x,z]= z -2 xy +\zeta, [y,z]= y^2 +u(x),$$
where $\zeta\in K$, $u(x)\in K[x]$.

\item When the characteristic of $K$ is $3$, we get the family with generators
$x,y,z$ and relations
$$[x,y]=z,\quad [x,z]= \alpha z +\delta\gamma\, xy + \delta y^2 + \big(\gamma x^3 +\delta \gamma^2 x^2-\gamma (\alpha-\delta \gamma)^2 x+\mu\big),$$
and
$$[y,z]= \gamma(\delta\gamma -\alpha)\,z,$$
where $\alpha,\delta,\gamma,\mu\in K$.


 \item When the characteristic of $K$ is $2$, we get the family with generators $x,y,z$ and relations
$$[x,y]=z, \quad [x,z]= \alpha z + \beta y + \big(\alpha\delta^{-1} x^3+\gamma_2 x^2 + \gamma_1 x + \gamma_0\big) +\delta y^2,$$
$$[y,z]=u(x)+\big(\alpha\delta^{-1} x^2+\gamma_1+\alpha\beta\delta^{-1}\big)y +\big(\alpha\delta^{-1} x+\gamma_2+\alpha\delta^{-1}\big)z,$$
where $\alpha\in\{0,1\}$, $\delta,\beta,\gamma_2,\gamma_1,\gamma_0\in K$ with $\delta\neq 0$, and $u(x)\in K[x]$ has degree at most $4$. 

\item When the characteristic of $K$ is $2$, we get the family with generators $x,y,z$ and relations
$$[x,y]=z, [x,z]=  z + \beta y + ( \gamma x^2 + \gamma x + \zeta ) +\delta y^2,  [y,z]= \gamma z,$$ where $\beta, \gamma, \zeta, \delta\in K$.

\item When the characteristic of $K$ is $2$, we get the 
family with generators $x,y,z$ and relations
\[
[x,y]=z,
\]
\[
[x,z]
=
\epsilon z+\beta y+
(\gamma_3x^3+\gamma_2x^2+\gamma_1x+\gamma_0),
\]
and
\[
[y,z]
=
r(x)
+
(\gamma_3x^2+\gamma_1+\beta\gamma_3)y
+
(\gamma_3x+\gamma_2+\epsilon\gamma_3)z
+
\epsilon y^2,
\]
where
\[
r(x)\in K[x]
\]
is arbitrary and
\[
\beta,\gamma_0,\gamma_1,\gamma_2,\gamma_3,\epsilon\in K.
\]

\item When the characteristic of $K$ is $2$, we get the 
family with generators $x,y,z$ and relations
\[
[x,y]=z,
\qquad
[x,z]=\epsilon z+q(x),
\]
and
\[
[y,z]
=
r(x)+q^{[1]}(x)y+H_\epsilon(q)(x)z+\epsilon y^2,
\]
where $q(x),r(x)\in K[x]$ are arbitrary,
\[
H_\epsilon(q)(x)
:=
\sum_{n\geq 2}\epsilon^{\,n-2}q^{[n]}(x),
\]
and $q^{[n]}(x)$ denotes the coefficient of $t^n$ in
$q(x+t)$.


\item When the characteristic of $K$ is $2$, we get the family with generators $x,y,z$ and relations
$$[x,y]=z, [x,z]=\epsilon z, [y,z]= u(x)+v(x)(\epsilon y+z)+\epsilon y^2,$$
where $\epsilon\in K, u(x),v(x)\in K[x]$.
\end{enumerate} 
\end{thm}
\begin{remark}
In family (1), ``filtration-compatible'' means that the Ore
presentation admits positive weights on $x,y,z$ for which every
correction term in an Ore relation has weight strictly smaller than
the corresponding quadratic leading term. After the reweightings described in Section \ref{sec:prelim} these need not
be the weights of the filtration originally given. We do not claim that an arbitrary
iterated Ore extension satisfying only the stated unipotence
condition necessarily has this property.

We note that although we are using the word ``classification'' we do not claim that this list is irredundant up to isomorphism, and so the isomorphism problem within and among families is not resolved; moreover, when $K$ is algebraically closed, family (2) can also be realized as an iterated Ore extension by making a change of variables $x'=x, y'\in Ky+Kz, z'=y$ with $y'$ a linear combination of $y$ and $z$ (involving $z$ non-trivially) that makes $y'$ an eigenvector for ${\rm ad}_x$.  In particular, when $K$ is algebraically closed and of characteristic zero, the classification consists precisely of families (1), (3), and (4).
\end{remark}

The outline of this paper is as follows. In \S \ref{sec:prelim} we establish our notation and the cases we study to obtain the classification result, Theorem \ref{thm:main}. In \S\ref{sec:I}--\ref{sec:IIIb} we work through each of the cases and obtain the families given in Theorem \ref{thm:main}. In \S\ref{sec:main} and \S\ref{sec:et} we prove Theorems \ref{thm:main} and \ref{thm:main2} respectively.

\section{Preliminaries for proof of Theorem \ref{thm:main}}\label{sec:prelim}
We let $K$ be a field and let $A$ be a filtered deformation of $K[x,y,z]$.  In particular, $A$ is equipped with an ascending filtration 
$$K=V_0\subseteq V_1\subseteq V_2\subseteq \cdots$$ 
such that each $V_i$ is finite-dimensional,
\[
A=\bigcup_{n\geq0}V_n,
\qquad
V_iV_j\subseteq V_{i+j}
\quad(i,j\geq0),
\]
\[
\operatorname{gr}A
=
\bigoplus_{n\geq0}V_n/V_{n-1},
\qquad V_{-1}:=0,
\]
is isomorphic to $K[x,y,z]$. Given a nonzero element $u\in A$, we let ${\rm deg}(u)$ denote the smallest nonnegative integer $n$ such that $u\in V_n$, and call this the degree of $u$.  We take $-\infty$ to be the degree of $0$.  Since the associated graded ring of $A$ is commutative, we have the inequality
$${\rm deg}([u,v])\le {\rm deg}(u)+{\rm deg}(v)-1$$ for $u,v\in A$.  Since the associated graded ring of $A$ is a polynomial algebra in three variables, we have that $A$ has generators $x,y,z$ and $A$ has a $K$-basis given by monomials $x^i y^j z^k$ with $i,j,k\ge 0$.

We let $a,b$, and $c$ denote the respective degrees of $x,y$, and $z$ and we may assume that $a\le b\le c$.  Then since $[x,y]$ has degree at most $a+b-1$, we see that $[x,y]$ is a linear combination of monomials $x^i y^j z^k$ in which either $k=1$ and $i=j=0$; or $j=1$ and $i=k=0$; or $j=k=0$.  In particular, if $z$ occurs with nonzero coefficient in our expression for $[x,y]$ as a linear combination of monomials $x^i y^j z^k$, then we have
$[x,y]=\alpha z+\beta y+f(x)$ for some $\alpha,\beta \in K$ with $\alpha\neq 0$ and some polynomial $f(x)$ in $K[x]$. 
Put
\[
c_*:=a+b-1.
\]
The degree bounds on the remaining commutators give
\[
[x,z]=\eta z+P(x,y),
\qquad
[y,z]=v(x)z+Q(x,y),
\]
where every monomial of $P$ has weight at most $a+c-1$, every
monomial of $Q$ has weight at most $b+c-1$, and
\[
a\deg v\leq b-1.
\]
First increase the weight of the old generator $z$ from $c$ to $c_*$.
The polynomial parts $P$ and $Q$ still have weight strictly smaller
than the corresponding quadratic terms, while
\[
\deg(v(x)z)\leq b+c_*-1.
\]
Thus all three commutator relations still strictly lower weight, and
the resulting filtration again has polynomial associated graded ring.

Now put
\[
Z:=\alpha z+\beta y+f(x).
\]
The element $Z$ has weight $c_*$, and its leading form is
\[
\overline Z=\alpha\overline z+
P_0(\overline x,\overline y)
\]
for some polynomial $P_0$.  Hence replacing $\overline z$ by
$\overline Z$ is an invertible triangular change of polynomial
coordinates in the associated graded ring.  After renaming $Z$ as
$z$ and $c_*$ as $c$, we may therefore assume that
\[
[x,y]=z,
\]
with $x,y,z$ again giving polynomial generators of the associated
graded ring.

Alternatively, if $z$ does not occur in our expression for $[x,y]$ then there are two possibilities:
\begin{itemize}
\item either $[x,y]=\alpha y+h(x)$ for some nonzero
$\alpha\in K$ and some polynomial $h(x)$.  In this case, first increase
the weight of the old generator $y$, if necessary, so that
\[
B:=\max\{b,a\deg h\}
\]
is its new weight, and then increase the weight of $z$ to some
$C\geq B$ sufficiently large that the polynomial parts of the two
remaining commutator relations still have weight strictly smaller than
the corresponding quadratic terms.  Terms involving $z$ continue to
have strictly smaller weight: in $[x,z]$ the coefficient of $z$ is
scalar, while in $[y,z]$ a term $x^iz$ necessarily satisfies
$ai\leq b-1\leq B-1$.  Thus this again defines a filtration with
polynomial associated graded ring.

Now put
\[
y'=\alpha y+h(x),\qquad x'=\alpha^{-1}x.
\]
The leading form of $y'$ is an invertible triangular change of the
polynomial coordinate $\bar y$, and
\[
[x',y']=y'.
\]
After renaming the generators and their weights, we may therefore
assume that $[x,y]=y$; or

\item $[x,y]=h(x)$ for some polynomial $h(x)$.
\end{itemize}
Thus after making these substitutions we may assume that our algebra has three generators $x,y,z$ with respective degrees $a,b,c$ and $a\le b\le c$ such that one of the following must occur:
\begin{enumerate}
\item[Case I:] $[x,y]=h(x)$ for some $h(x)\in K[x]$;
\item[Case II:] $[x,y]=y$;
\item[Case III:] $[x,y]=z$.
\end{enumerate}
As it turns out, Case III is the most difficult case to study, with the other cases being considerably simpler.  We thus classify filtered deformations in Cases I and II first and divide Case III into two subcases:
\begin{enumerate}
\item[Case III(a):] $[x,y]=z$, and $b>2a-2$.
\item[Case III(b):] $[x,y]=z$ and $b\le 2a-2$.
\end{enumerate}

\section{Filtered deformations in Case I}\label{sec:I}
We recall that we have a $K$-algebra generated by elements $x,y,z$ with respective degrees $a,b,c$ and $a\le b\le c$ and such that $[x,y]=h(x)$.  This means that 
the subalgebra $B$ generated by $x$ and $y$ is
$K[x][y;\delta]$, where $\delta(x)=-h(x)$.
 Now since $x^i y^j z^k$ has degree strictly larger than $a+c-1$ if $k\ge 1$ and $i+j\ge 1$, we see that 
$[x,z]$ can only involve the monomials $z$ and monomials $x^i y^j$, thus $[x,z]=\alpha z+ h(x,y)$ for some $\alpha\in K$ and some $h(x,y)\in B$.  Similarly, $[y,z]$ can only involve monomials of the form $x^i z$ with $i\ge 0$ and monomials of the form $x^i y^j$ and so $[y,z]=u(x) z+ g(x,y)$ for some $u(x)\in K[x]$ and some $g(x,y)\in B$.  

Define
\[
\sigma(x)=x-\alpha,
\qquad
\sigma(y)=y-u(x),
\]
and
\[
\mu(x)=-h(x,y),
\qquad
\mu(y)=-g(x,y).
\]

The Jacobi identity
\[
[[x,y],z]=[[x,z],y]+[x,[y,z]]
\]
has zero $z$-coefficient on the right, since the contributions
$-\alpha u(x)z$ and $\alpha u(x)z$ cancel.  On the other hand,
from
\[
zx=(x-\alpha)z-h(x,y)
\]
we obtain
\[
zh(x)=h(x-\alpha)z+\text{terms in }B,
\]
and hence the $z$-coefficient on the left is
\[
h(x)-h(x-\alpha).
\]
Thus
\[
h(x)=h(x-\alpha).
\]
This is precisely the condition that $\sigma$ preserve the defining
relation
\[
yx-xy+h(x)=0
\]
of $B=K[x][y;\delta]$, where $\delta(x)=-h(x)$. Thus $\sigma$
induces an automorphism of $B$. It is unipotent because it is
triangular and $\sigma-\operatorname{Id}$ is locally nilpotent.

After the coefficient of $z$ is removed, the remaining part of the
Jacobi identity is precisely the condition that the
$\sigma$-derivation determined by $\mu(x)$ and $\mu(y)$ annihilate
the defining relation of $B$. Hence $\mu$ induces a
$\sigma$-derivation of $B$. Since
\[
zx=\sigma(x)z+\mu(x),
\qquad
zy=\sigma(y)z+\mu(y),
\]
we obtain
\[
A=K[x][y;\delta][z;\sigma,\mu].
\]
This is an iterated Ore extension of unipotent type and hence
belongs to family (1) in Theorem~\ref{thm:main}.

\section{Filtered deformations in Case II}\label{sec:II}

Here the subalgebra $B$ generated by $x$ and $y$ has defining
relation
\[
[x,y]=y.
\]
Thus
\[
B=K[y][x;\partial],
\qquad
\partial(y)=y.
\]
As in Case I, the degree bounds give
\[
[x,z]=\alpha z+h(x,y)
\]
and
\[
[y,z]=u(x)z+g(x,y),
\]
where $\alpha\in K$, $u(x)\in K[x]$, and $h(x,y),g(x,y)\in B$.

We first show that $u(x)=0$. The Jacobi identity gives
\[
[[x,y],z]=[[x,z],y]+[x,[y,z]].
\]
The left-hand side is
\[
[[x,y],z]=[y,z]=u(x)z+g(x,y).
\]
On the right-hand side, the contributions containing $z$ from
$\alpha[z,y]$ and $u(x)[x,z]$ cancel, while every remaining term
belongs to $B$. Comparison of the coefficient of $z$ therefore
gives
\[
u(x)=0.
\]
Hence
\[
[y,z]=g(x,y).
\]

Define a map $\sigma\colon B\to B$ by
\[
\sigma(x)=x-\alpha,
\qquad
\sigma(y)=y.
\]
It preserves the defining relation, since
\[
[\sigma(x),\sigma(y)]
=
[x-\alpha,y]
=
y
=
\sigma(y).
\]
Thus $\sigma$ is an automorphism of $B$; moreover,
$\sigma-\operatorname{Id}$ is locally nilpotent, so $\sigma$ is
unipotent.

Define
\[
\mu(x)=-h(x,y),
\qquad
\mu(y)=-g(x,y).
\]
The remaining $B$-valued part of the Jacobi identity is precisely
the condition that the $\sigma$-derivation determined on the free
algebra by these formulas annihilates the defining relation
\[
xy-yx-y.
\]
Consequently $\mu$ induces a $\sigma$-derivation of $B$. The
relations
\[
zx=\sigma(x)z+\mu(x),
\qquad
zy=\sigma(y)z+\mu(y)
\]
are exactly the original commutator relations. Therefore
\[
A=B[z;\sigma,\mu]
\]
is an iterated Ore extension of unipotent type and belongs to
family (1).

\section{Filtered deformations in Case III(a)}\label{sec:IIIa}
In this case $c\le a+b-1$, so $[x,z]$ has degree at most $2a+b-2$. In particular, since this is less than $2b$ we see that we cannot have any monomials $x^i y^j z^k$ with $j\ge 2$ or with $k\ge 1$ unless $(i,j,k)=(0,0,1)$. Thus 
$[x,z]= u(x) +v(x) y + \alpha z$ for some polynomials $u(x), v(x)\in K[x]$ and $\alpha\in K$. Moreover, $v(x)$ is at most linear. Similarly, 
$[y,z]$ has degree at most $b+c-1\le 2b+a-2$ and so $[y,z]$ can only involve monomials $y^2$ and monomials of the form $x^i, x^j y, x^k z$.
Thus we see in this case there are constants $\alpha,\beta,\gamma,\delta \in K$ and polynomials $u(x), f(x), g(x), h(x)$
such that we have relations of the form
$$[x,y]=z, \quad [x,z] = u(x) + (\beta x+\gamma)y + \alpha z, \quad [y,z]=f(x)+g(x)y + h(x)z+\delta y^2.$$

\subsection{Subcase 1: $(\beta,\gamma)\neq (0,0)$.}
In this case, we can write $u(x) = (\beta x + \gamma) v(x) + \zeta$ for some polynomial $v(x)$ and some $\zeta\in K$, where $\zeta=0$ if $\beta=0$. We can make a substitution 
$$x' = x, \quad y' = y+v(x), \quad z' = z,$$ 
and assume that
$$[x,y]=z, \quad [x,z] = \zeta + (\beta x+\gamma) y + \alpha z, \quad [y,z] = f(x) + g(x) y+h(x)z+\delta y^2.$$

The substitution is an invertible change of algebra generators.  Moreover,
the left $K[x]$-module
\[
K[x]+K[x]y+K[x]z
\]
is preserved by $\operatorname{ad}_x$, and hence
$[v(x),z]$ belongs to this module.  Thus the relation for $[y,z]$
retains the displayed form, with possibly different polynomial
coefficients.  We do not claim that this change preserves the original
filtration; compatible filtrations for the resulting presentations are
verified independently in Section~\ref{sec:main}.

Then the Jacobi identity constraint gives
$$[[x,y],z] = [[x,z],y]+[x,[y,z]].$$
Since $[[x,y],z] = [z,z] = 0$, using our relations we see we must have
$$\beta zy + \alpha [z,y] + g(x) z + h(x) [x,z] + \delta yz + \delta zy = 0,$$
which simplifies to
\begin{equation}
\label{eq:bdg}
(\beta+2\delta) yz - (\alpha +\beta+\delta) [y,z] + g(x) z + h(x)(\zeta + (\beta x + \gamma)y + \alpha z) = 0.
\end{equation}
Then considering the coefficient of $yz$ on both sides, we see that $\beta=-2\delta$. 
Considering the coefficient of $y^2$ on both sides, we see that $(\alpha+\beta+\delta) \delta=(\alpha-\delta)\cdot \delta=0$, so either $\delta=0$ or $\alpha=\delta$.  
In particular, if $\delta=0$, then $\beta=0$ and so our commutators simplify to:
$$[x,y]=z, \quad [x,z] = \gamma y +\alpha z, \quad [y,z]=f(x)+g(x)y + h(x) z.$$
The Jacobi identity $0 = [[x,z],y] + [x,[y,z]]$ now yields:
$$0 = [\gamma y + \alpha z, y] + [x, f(x) + g(x)y + h(x)z],$$
which expands to
$$0 = \alpha [z,y] + g(x)[x,y] + h(x)[x,z].$$
Substituting the commutators gives:
$$0 = -\alpha [y,z] + g(x)z + h(x)(\gamma y + \alpha z),$$
or equivalently, moving the $\alpha[y,z]$ term to the left hand side:
$$\alpha (f(x) + g(x)y + h(x)z) = \gamma h(x)y + (g(x) + \alpha h(x))z.$$
Comparing coefficients as modules over $K[x]$ then gives
$$\alpha h(x) = g(x) + \alpha h(x) \implies g(x) = 0.$$
Looking at the $y$ coefficient gives $\alpha g(x) = \gamma h(x)$. Since $g(x) = 0$, this becomes $0 = \gamma h(x)$. By the assumption of this subcase, $(\beta,\gamma)\neq (0,0)$. Having already shown $\beta=0$, it follows that $\gamma \neq 0$, which forces $h(x)=0$. 

Comparing the $x$-polynomial coefficients gives $\alpha f(x) = 0$. Thus, our commutators simplify to:
$$[x,y] = z, \quad [x,z] = \gamma y + \alpha z, \quad [y,z] = f(x),$$
with the constraint that either $\alpha = 0$ or $f(x) = 0$.
If $f(x)=0$, then our algebra is an iterated Ore extension
$K[y,z][x;\partial]$, where
\[
\partial(y)=z,
\qquad
\partial(z)=\gamma y+\alpha z.
\]
Otherwise $f(x)
\neq0$, and the constraint $\alpha f(x)=0$ gives
$\alpha=0$. We then obtain
\begin{equation}
[x,y]=z,\qquad
[x,z]=\gamma y,\qquad
[y,z]=f(x),
\end{equation}
which is family (2) in Theorem~\ref{thm:main}.

Thus when $\delta=0$ we see this case yields families from (1) and (2) in Theorem \ref{thm:main}.

When $\delta\neq 0$, we have $\alpha=\delta, \beta=-2\delta$ so our commutators become
$$[x,y]=z, \quad [x,z] = \zeta + (-2\delta x+\gamma) y + \delta z, \quad [y,z] = f(x) + g(x) y+h(x)z+\delta y^2.$$

If the characteristic of $K$ is not $2$, then $2\delta
\neq0$. Make
the change of generators
\[
X=\delta^{-1}x-\frac{\gamma}{2\delta^2},
\qquad
Y=y,
\qquad
Z=\delta^{-1}z.
\]
Then
\[
[X,Y]=Z.
\]
Since
\[
x=\delta X+\frac{\gamma}{2\delta},
\qquad
z=\delta Z,
\]
the relation for $[X,Z]$ becomes
\[
[X,Z]
=
Z-2XY+\delta^{-2}\zeta.
\]
The relation for $[Y,Z]$ retains the form
\[
[Y,Z]
=
Y^2+\widetilde f(X)+\widetilde g(X)Y+\widetilde h(X)Z
\]
for suitable transformed polynomials
$\widetilde f,\widetilde g,\widetilde h\in K[X]$.
After renaming $X,Y,Z$, the scalar $\delta^{-2}\zeta$, and the
transformed polynomials, our relations have the form

$$[x,y]=z, \quad [x,z] = -2xy + z+\zeta, \quad [y,z] = f(x) + g(x) y+h(x)z+y^2.$$
Since $\alpha+\beta+\delta=0$ and since $\beta x+ \gamma$ is nonzero, looking at the $x$-polynomial coefficients of $y$ on both sides of Equation (\ref{eq:bdg}) gives
$h(x)=0$.  Then, looking at the $x$-polynomial coefficients of $z$ gives $g(x)+\alpha h(x)=0$ and since $\alpha=\delta\neq 0$, we see $g(x)=0$, so we get 
the commutators
$$[x,y]=z,\quad [x,z] =z-2xy +\zeta, \quad [y,z]=y^2+f(x),$$ which is family (4) from Theorem \ref{thm:main}.

Finally, if the characteristic of $K$ is two and $\delta\neq 0$, then
$\beta=-2\delta=0$, so by assumption $\gamma\neq 0$, and
since $\beta=0$, we also have $\zeta=0$. Thus
\[
[x,y]=z,\qquad
[x,z]=\gamma y+\delta z,\qquad
[y,z]=\delta y^2+f(x)+g(x)y+h(x)z.
\]
The Jacobi identity gives
\[
0=[x,[y,z]]+[[x,z],y]
 =\bigl(\delta[y,z]+g(x)z+h(x)[x,z]\bigr)
   +\delta[z,y],
\]
and hence
\[
g(x)z+h(x)(\gamma y+\delta z)=0.
\]
Since $\gamma\neq0$, it follows that $g(x)=h(x)=0$. We therefore
obtain
\[
[x,y]=z,\qquad
[x,z]=\gamma y+\delta z,\qquad
[y,z]=\delta y^2+f(x),
\]
where $f(x)\in K[x]$ is arbitrary. This is family (8), with
$\epsilon=\delta$, $\beta=\gamma$ and
$\gamma_3=\gamma_2=\gamma_1=\gamma_0=0$.



\subsection{Subcase 2: $\beta=\gamma=0$.}
In this case, our relations are
\[
[x,y]=z,\qquad
[x,z]=u(x)+\alpha z,\qquad
[y,z]=f(x)+g(x)y+h(x)z+\delta y^2.
\]
The Jacobi identity gives
\[
[x,[y,z]]=-[[x,z],y].
\]
Now
\[
[x,y^2]=zy+yz=2yz-[y,z].
\]
Observe that the Hasse--Schmidt expansion gives
\begin{equation} \label{eq:u(x)y}
[u(x),y]
=
\sum_{n=1}^{\infty}
(-1)^{n-1}u^{[n]}(x)\operatorname{ad}_x^n(y),
\end{equation}
where $u^{[n]}(x)$ denotes the coefficient of $t^n$ in
$u(x+t)$. Since $\operatorname{ad}_x^n(y)$ lies in
$K[x]+K[x]z$ for every $n\geq 1$, we see that $[u(x),y]$ is a linear combination of terms of the form
$x^i$ and $x^iz$, and so comparison of the coefficient of $yz$ in the
Jacobi identity above gives
\[
2\delta=0.
\]

Suppose first that $\delta=0$. Then our relations are
\[
[x,y]=z,\qquad
[x,z]=u(x)+\alpha z,\qquad
[y,z]=f(x)+g(x)y+h(x)z.
\]
The coefficients of terms of the form $x^iy$ in $[x,[y,z]]$
are zero, whereas using Equation (\ref{eq:u(x)y}) gives that the contribution of such terms in
$-[[x,z],y]$ is $\alpha g(x)y$.
Hence
\[
\alpha g(x)=0.
\]

If $g(x)=0$, then our algebra is given by the relations
\[
[x,y]=z,\qquad
[x,z]=u(x)+\alpha z,\qquad
[y,z]=f(x)+h(x)z,
\]
and so it is an iterated Ore extension of unipotent type of
the form
\[
K[x][z;\sigma_1,\delta_1][y;\delta_2].
\]
On the other hand, if $\alpha=0$, then our algebra is of the
form
\[
[x,y]=z,\qquad
[x,z]=u(x),\qquad
[y,z]=f(x)+g(x)y+h(x)z,
\]
which is again an iterated Ore extension of unipotent type,
of the form
\[
K[x][z;\delta_1][y;\sigma_2,\delta_2].
\]
Thus when $\delta=0$ we obtain algebras from family (1) of
Theorem~\ref{thm:main}.

It remains to consider the case $\delta\neq 0$. Since
$2\delta=0$, we have the characteristic of $K$ is two. In this
case
\[
[x,y^2]=zy+yz=[y,z],
\]
and the Jacobi identity becomes
\[
[x,[y,z]]=[[x,z],y].
\]
Expanding both sides gives
\[
g(x)z+h(x)\bigl(u(x)+\alpha z\bigr)
+\delta[y,z]
=
[u(x),y]+\alpha[y,z],
\]
and hence
\begin{equation}
\label{eq:IIIa-subcase2-char2}
[u(x),y]
=
h(x)u(x)+(g(x)+\alpha h(x))z
+(\delta+\alpha)[y,z].
\end{equation}
The left-hand side belongs to $K[x]+K[x]z$. Comparing the
coefficient of $y^2$ in
Equation~\eqref{eq:IIIa-subcase2-char2} therefore gives
\[
(\delta+\alpha)\delta=0.
\]
Since $\delta\neq0$, we have
\[
\alpha=\delta,
\]
and consequently
\begin{equation}
\label{eq:IIIa-subcase2-char2-reduced}
[u(x),y]
=
h(x)u(x)+(g(x)+\delta h(x))z.
\end{equation}

For $n\geq2$, an induction gives that
\[
\operatorname{ad}_x^n(y)
=
\delta^{n-2}u(x)+\delta^{n-1}z.
\]
We now define
\[
H_\delta(u)
=
\sum_{n\geq2}\delta^{n-2}u^{[n]}(x),
\]
where, as above, $u^{[n]}(x)$ is the coefficient of $t^n$
in $u(x+t)$. The sum is finite since $u$ is a polynomial. Since $K$ has characteristic two, the
Hasse--Schmidt expansion now gives
\[
[u(x),y]
=
u(x)H_\delta(u)
+
\bigl(u^{[1]}(x)+\delta H_\delta(u)\bigr)z.
\]
Comparing this with
Equation~\eqref{eq:IIIa-subcase2-char2-reduced}, we obtain
\begin{equation}
\label{eq:IIIa-Hasse-constraints}
u(x)\bigl(h(x)-H_\delta(u)\bigr)=0
\end{equation}
and
\begin{equation}
\label{eq:IIIa-Hasse-constraints2}
g(x)+\delta h(x)
=
u^{[1]}(x)+\delta H_\delta(u).
\end{equation}

Suppose first that $u(x)\neq0$. Since $K[x]$ is a domain,
Equation~\eqref{eq:IIIa-Hasse-constraints} gives
\[
h(x)=H_\delta(u),
\]
and then Equation~\eqref{eq:IIIa-Hasse-constraints2} gives
\[
g(x)=u^{[1]}(x).
\]
Thus in this case our algebra has relations
\[
[x,y]=z,
\]
\[
[x,z]=u(x)+\delta z,
\]
and
\[
[y,z]
=
f(x)+u^{[1]}(x)y+H_\delta(u)z+\delta y^2.
\]
This is the characteristic-two family given in Theorem \ref{thm:main} (9) to allow an arbitrary polynomial in the
$[x,z]$ relation.

Finally, suppose that $u(x)=0$. Then
Equation~\eqref{eq:IIIa-Hasse-constraints2} gives
\[
g(x)=\delta h(x),
\]
and hence our relations are
\[
[x,y]=z,\qquad
[x,z]=\delta z,
\]
and
\[
[y,z]
=
f(x)+h(x)(\delta y+z)+\delta y^2.
\]
These are precisely the algebras in family (10). This completes Case~III(a).




\section{Filtered deformations in Case III(b)} \label{sec:IIIb}
Here since $b\le 2a-2$ and $z=[x,y]$ we see that $c\le a+b-1\le 3a-3$.  Thus
$[x,z]$ has degree at most $2a+b-2\le 4a-4$ and $[y,z]$ has degree at most $a+2b-2\le 5a-6$.  In particular, $[x,z]$ is a linear combination of monomials $x^i y^j z^k$ from
$$\{1,x,x^2,x^3, y, xy, y^2, z\}$$ and 
$[y,z]$ is a linear combination of monomials $x^i y^j z^k$ from
$$\{1,x,x^2,x^3, x^4, y, xy, x^2y, y^2, z, xz\}.$$  
So we can write
\begin{equation}
[x,z]= \alpha z + (\beta_1 x +\beta_0) y + (\gamma_3 x^3+ \gamma_2 x^2 + \gamma_1 x + \gamma_0 ) +\delta y^2
\end{equation}
and
\begin{equation}
[y,z]= u_0(x) + u_1(x) y + u_2(x) z + \epsilon y^2 
\end{equation}
where $u_0, u_1, u_2$ have degrees at most $4, 2$, and $1$ respectively.
We write $$u_0(x)=\textstyle\sum_{i=0}^4 c_ix^i,\ u_1(x)=e_2x^2+e_1x+e_0,\ u_2(x)=f_1x+f_0.$$

This is the most involved case. We first divide according to whether
$\delta=0$. When $\delta=0$, we distinguish $\epsilon=0$ from
$\epsilon
\neq0$, giving Subcases 1 and 2 below. When $\delta
\neq0$,
we first eliminate the coefficient $\epsilon$ by a change of
generators. Only after making that change do we divide according to
whether the resulting coefficients $\alpha$ and $\beta_1$ are zero.
This gives the four mutually exclusive Subcases 3--6 and ensures that
the case division is exhaustive.

For the analysis that follows, we will need these formulas:

\begin{equation} \label{eq:xxy}
[x^2,y] = xz+zx = 2xz - [x,z],
\end{equation}
\begin{equation} \label{eq:xxxy}
 [x^3, y] = x^2 z + xzx + zx^2 = 3x^2z - 3x[x,z] + [x,[x,z]].
 \end{equation}

\subsection{Subcase 1: $\delta=\epsilon=0$.}
In this case, the commutator brackets take the form of those given in Case III(a), and so we obtain no new deformations in this case.

\subsection{Subcase 2: $\delta=0$, $\epsilon
\neq 0$.}
Throughout this subcase we write
\[
P(x):=\gamma_3x^3+\gamma_2x^2+\gamma_1x+\gamma_0,
\qquad
b(x):=\beta_1x+\beta_0,
\]
so that the relations under consideration are
\begin{equation}
\label{eq:IIIb-sub2-relations}
[x,y]=z,
\qquad
[x,z]=\alpha z+b(x)y+P(x),
\end{equation}
\[
[y,z]=u_0(x)+u_1(x)y+u_2(x)z+\epsilon y^2,
\]
where, as above, $u_1(x)=e_2x^2+e_1x+e_0$ and $u_2(x)=f_1x+f_0$.
Since $[[x,y],z]=[z,z]=0$, the Jacobi identity gives
\begin{equation}
\label{eq:IIIb-sub2-jacobi}
[[x,z],y]+[x,[y,z]]=0.
\end{equation}
We expand the left-hand side of Equation~\eqref{eq:IIIb-sub2-jacobi}
using Equations (\ref{eq:xxy}) and (\ref{eq:xxxy}), write the result
in the basis $x^iy^jz^k$, and compare coefficients.

The coefficients of $yz$ and of $y^2$ are
\[
\beta_1+2\epsilon
\qquad\text{and}\qquad
-\epsilon(\alpha+\beta_1+\epsilon),
\]
respectively. Since $\epsilon
\neq 0$, we obtain
\begin{equation}
\label{eq:IIIb-sub2-leading}
\beta_1=-2\epsilon
\qquad\text{and}\qquad
\alpha=\epsilon.
\end{equation}
In particular $\alpha+\beta_1+\epsilon=0$, and the coefficient with
which $u_0(x)$ enters Equation~\eqref{eq:IIIb-sub2-jacobi} is exactly
$-(\alpha+\beta_1+\epsilon)$. Hence no condition whatsoever is imposed
on $u_0(x)$ in this subcase.

Now put
\[
T:=f_1-3\gamma_3,
\qquad
L:=\epsilon\gamma_3+f_0-\gamma_2,
\qquad
Q(x):=Tx+L.
\]
After substituting Equation~\eqref{eq:IIIb-sub2-leading}, the remaining
content of Equation~\eqref{eq:IIIb-sub2-jacobi} is the single identity
\begin{equation}
\label{eq:IIIb-sub2-nf}
b(x)Q(x)\,y+P(x)Q(x)+R(x)\,z=0,
\end{equation}
where
\[
R(x)
=
(e_2+3\gamma_3)x^2
+
(e_1+\epsilon f_1-5\epsilon\gamma_3+2\gamma_2)x
+
\bigl(e_0+\beta_0\gamma_3+\epsilon L+\gamma_1\bigr).
\]
Thus Equation~\eqref{eq:IIIb-sub2-jacobi} is equivalent to the three
conditions
\begin{equation}
\label{eq:IIIb-sub2-conditions}
b(x)Q(x)=0,
\qquad
P(x)Q(x)=0,
\qquad
R(x)=0,
\end{equation}
the last of which says precisely that
\begin{equation}
\label{eq:IIIb-sub2-e}
e_2=-3\gamma_3,
\qquad
e_1=-\epsilon f_1+5\epsilon\gamma_3-2\gamma_2,
\qquad
e_0=-\beta_0\gamma_3-\epsilon L-\gamma_1.
\end{equation}
We now consider the three possibilities permitted by
Equation~\eqref{eq:IIIb-sub2-conditions}.

\medskip

\noindent\textbf{The characteristic of $K$ is not $2$.}
Here $b(x)=-2\epsilon x+\beta_0$ has degree exactly one, so, since
$K[x]$ is a domain, the first condition in
Equation~\eqref{eq:IIIb-sub2-conditions} forces
\[
Q(x)=0,
\qquad\text{that is,}\qquad
f_1=3\gamma_3
\quad\text{and}\quad
f_0=\gamma_2-\epsilon\gamma_3.
\]
Write $P(x)=b(x)v(x)+\zeta$ with $v(x)\in K[x]$ and $\zeta\in K$, and
make the substitution
\[
x'=x,
\qquad
y'=y+v(x),
\qquad
z'=z.
\]
Then $[x',y']=z'$ and
\[
[x',z']=\epsilon z'+b(x')y'+\zeta.
\]
Moreover $[y',z']=[y,z]+[v(x),z]$, and since
$\operatorname{ad}_x^n(z)\in K[x]+K[x]y+K[x]z$ for every $n\geq 1$, the
Hasse--Schmidt expansion gives
$[v(x),z]\in K[x]+K[x]y+K[x]z$. Consequently the relations retain the
shape of Equation~\eqref{eq:IIIb-sub2-relations}, with the same
$\alpha=\epsilon$, $\beta_1=-2\epsilon$ and $\beta_0$, and with $P(x)$
replaced by the constant $\zeta$.

We may therefore apply Equations~\eqref{eq:IIIb-sub2-conditions} and
\eqref{eq:IIIb-sub2-e} to the new coefficients. Since the new
$\gamma_3,\gamma_2,\gamma_1$ all vanish, the new $T$ and $L$ are just
the new $f_1$ and $f_0$; as $b(x)
\neq 0$ we again get $Q(x)=0$, and so
the new $u_2$ is zero. Equation~\eqref{eq:IIIb-sub2-e} then gives that
the new $u_1$ is zero as well. Hence, after renaming, our relations are
\[
[x,y]=z,
\qquad
[x,z]=\epsilon z+b(x)y+\zeta,
\qquad
[y,z]=u_0(x)+\epsilon y^2.
\]
Finally make the invertible change of generators
\[
X=\epsilon^{-1}x-\frac{\beta_0}{2\epsilon^2},
\qquad
Y=y,
\qquad
Z=\epsilon^{-1}z.
\]
Since $b(x)=-2\epsilon^2X$, we obtain
\[
[X,Y]=Z,
\qquad
[X,Z]=Z-2XY+\epsilon^{-2}\zeta,
\qquad
[Y,Z]=Y^2+\widetilde u(X)
\]
for a suitable $\widetilde u(X)\in K[X]$. After renaming $X,Y,Z$, the
scalar $\epsilon^{-2}\zeta$ and the polynomial $\widetilde u$, this is
precisely family (4) in Theorem~\ref{thm:main}.

\medskip

\noindent\textbf{The characteristic of $K$ is $2$.}
Now Equation~\eqref{eq:IIIb-sub2-leading} gives $\beta_1=-2\epsilon=0$,
so $b(x)=\beta_0$ is a constant and
Equation~\eqref{eq:IIIb-sub2-conditions} becomes
\[
\beta_0\,Q(x)=0,
\qquad
P(x)Q(x)=0,
\qquad
R(x)=0.
\]

Suppose first that $\beta_0
\neq0$ or $P(x)
\neq0$. Since $K[x]$ is a
domain, we again get $Q(x)=0$, that is,
\[
f_1=\gamma_3
\qquad\text{and}\qquad
f_0=\gamma_2+\epsilon\gamma_3,
\]
and then Equation~\eqref{eq:IIIb-sub2-e} gives
\[
e_2=\gamma_3,
\qquad
e_1=0,
\qquad
e_0=\gamma_1+\beta_0\gamma_3.
\]
Thus
\[
[x,y]=z,
\qquad
[x,z]
=
\epsilon z+\beta_0y+
(\gamma_3x^3+\gamma_2x^2+\gamma_1x+\gamma_0),
\]
and
\[
[y,z]
=
u_0(x)
+
(\gamma_3x^2+\gamma_1+\beta_0\gamma_3)y
+
(\gamma_3x+\gamma_2+\epsilon\gamma_3)z
+
\epsilon y^2,
\]
which is family (8) in Theorem~\ref{thm:main} with $\beta=\beta_0$.

Finally suppose that $\beta_0=0$ and $P(x)=0$, so that
$\gamma_3=\gamma_2=\gamma_1=\gamma_0=0$. The first two conditions now
hold identically, and $T=f_1$ and $L=f_0$ are unconstrained, while
Equation~\eqref{eq:IIIb-sub2-e} reduces to
\[
e_2=0,
\qquad
e_1=\epsilon f_1,
\qquad
e_0=\epsilon f_0,
\]
that is, to $u_1(x)=\epsilon u_2(x)$. Hence
\[
[x,y]=z,
\qquad
[x,z]=\epsilon z,
\qquad
[y,z]=u_0(x)+u_2(x)\bigl(\epsilon y+z\bigr)+\epsilon y^2,
\]
which is family (10) in Theorem~\ref{thm:main}.

Since the three possibilities just treated exhaust
Equation~\eqref{eq:IIIb-sub2-conditions}, this completes Subcase 2.

\medskip

\noindent\textbf{Preliminary reduction when $\delta
\neq0$.}
We now assume that $\delta
\neq0$. Before dividing into further
subcases, we eliminate the coefficient $\epsilon$. Put
\[
P(x):=\gamma_3x^3+\gamma_2x^2+\gamma_1x+\gamma_0
\qquad\text{and}\qquad
t:=\epsilon\delta^{-1},
\]
and make the invertible change of generators
\[
X=x,\qquad Y=y-tx,\qquad Z=z.
\]
Then
\[
[X,Y]=Z.
\]
Since $y=Y+tX$ and $[X,Y]=Z$, we have
\begin{equation}
\label{eq:Y-plus-tX-square}
(Y+tX)^2=Y^2+2tXY-tZ+t^2X^2.
\end{equation}
Consequently the relation for $[X,Z]$ becomes
\[
[X,Z]
=
(\alpha-\epsilon)Z
+
\bigl((\beta_1+2\epsilon)X+\beta_0\bigr)Y
+
\widetilde{P}(X)
+
\delta Y^2,
\]
where
\[
\widetilde{P}(X)
=
P(X)
+
(t\beta_1+\delta t^2)X^2
+
t\beta_0X.
\]
In particular, the new coefficients of $Z$ and $XY$ are
\begin{equation}
\label{eq:alpha-beta-transform}
\alpha'=\alpha-\epsilon,
\qquad
\beta_1'=\beta_1+2\epsilon.
\end{equation}

Moreover,
\[
[Y,Z]=[y,z]-t[x,z]
=
\widetilde{u}_0(X)
+
\widetilde{u}_1(X)Y
+
\widetilde{u}_2(X)Z,
\]
where
\[
\widetilde{u}_1(X)
=
u_1(X)-t\beta_1X-t\beta_0,
\]
\[
\widetilde{u}_2(X)
=
u_2(X)-t\alpha,
\]
and
\[
\widetilde{u}_0(X)
=
u_0(X)+tXu_1(X)-tP(X)
-t^2\beta_1X^2-t^2\beta_0X.
\]
Thus the coefficient of $Y^2$ in $[Y,Z]$ is zero, while all the
degree bounds on $\widetilde{P}$ and the $\widetilde{u}_i$ are the
same as the original degree bounds.

We now replace $X,Y,Z$ by $x,y,z$ and suppress the primes and tildes
on all the transformed coefficients. Thus, throughout Subcases
3--6, we have
\[
\epsilon=0,
\]
and the symbols $\alpha,\beta_1,\beta_0,\gamma_i,u_i$ refer to the
coefficients obtained \emph{after} the preceding change of
generators.

For reference, when $\delta
\neq0$ and $\epsilon=0$, a direct
expansion of the Jacobi identity gives the following equivalent
system. Set
\[
M:=\alpha-\delta\gamma_3.
\]
Then
\begin{align}
\beta_1
&=-2\delta\gamma_3,
\label{eq:IIIb-compact-beta}\\
u_2(x)
&=\gamma_2-\alpha\gamma_3+3\gamma_3x,
\label{eq:IIIb-compact-u2}\\
u_1(x)
&=
M u_2(x)-3\gamma_3x^2
-(\gamma_3\beta_1+2\gamma_2)x
-\gamma_3\beta_0-\gamma_1,
\label{eq:IIIb-compact-u1}\\
M u_0(x)&=0,
\qquad
M u_1(x)=0.
\label{eq:IIIb-compact-M}
\end{align}
These identities hold in arbitrary characteristic.

There are now exactly four possibilities for the transformed
coefficients:
\[
\alpha\beta_1
\neq0,\qquad
\alpha=0,\ \beta_1
\neq0,\qquad
\alpha
\neq0,\ \beta_1=0,\qquad
\alpha=\beta_1=0.
\]
They are treated in Subcases 3--6, respectively.

\subsection{Subcase 3: $\alpha\beta_1\delta\neq 0$}

Here and in the remaining subcases, the coefficients are the
transformed coefficients obtained after eliminating $\epsilon$.
Thus $\epsilon=0$ before any normalization is performed.

Since $\alpha
\neq0$, make the change
\[
x'=\alpha^{-1}x,\qquad y'=\alpha y,\qquad z'=z,
\]
and rename the resulting coefficients, so that $\alpha=1$. Since
$\beta_1
\neq0$, we may then make the translation
\[
x'=x+\beta_1^{-1}\beta_0,\qquad y'=y,\qquad z'=z,
\]
and rename the resulting coefficients, so that $\beta_0=0$. These
changes preserve the conditions $\alpha
\neq0$ and $\beta_1
\neq0$.
We may therefore assume that our commutators are

$$[x,y]=z,\quad [x,z]= z + \beta_1 x y + (\gamma_3 x^3+ \gamma_2 x^2 + \gamma_1 x + \gamma_0 ) +\delta y^2$$
and
$$[y,z]= u_0(x) + u_1(x) y + u_2(x) z.$$
Then the Jacobi identity constraint from $0=[[x,y],z]$ gives $[[x,z],y]= -[x,[y,z]]$. 
We compute $[[x,z],y]$ using Equations (\ref{eq:xxy}) and (\ref{eq:xxxy}) and find
$$[[x,z],y] = [z,y] + \beta_1 zy + \gamma_3 (3x^2 z-3x[x,z]+[x,[x,z]]) + \gamma_2 (2xz -[x,z]) + \gamma_1 z.$$
While $$-[x,[y,z]]= -u_1(x) z - u_2(x) (z+\beta_1 xy +(\gamma_3 x^3+ \gamma_2 x^2 + \gamma_1 x + \gamma_0 ) +\delta y^2).$$
Simplifying and comparing coefficients we obtain the following constraints:
$$\beta_1 =- 2\delta\gamma_3, \quad f_1 = 3\gamma_3, \quad  f_0 = \gamma_2 - \gamma_3 $$
 $$e_2 = -3\gamma_3, \quad e_1 = 3\gamma_3 M - 2\gamma_2 - \gamma_3\beta_1, 
    e_0 = (\gamma_2-\gamma_3)M - \gamma_1,$$
    and $$M e_2  = M e_1 = M e_0=0,\quad M \cdot u_0(x) =0,$$
   where $ M = 1-\delta\gamma_3$.
   
In particular, since $\beta_1\neq 0$ is a multiple of $2$, the characteristic of $K$ cannot be equal to $2$ in this subcase, so we assume for the remainder of this subcase that $K$ does not have characteristic $2$.

If $M\neq 0$ and the characteristic of $K$ is not $3$, then the constraints give $u_0(x)=e_2=e_1=e_0=\beta_1=\gamma_3=\gamma_2=\gamma_1=f_1=f_0=0$, and this is not allowed since we are assuming that $\beta_1\neq 0$.

If $M\neq 0$ and the characteristic of $K$ is $3$, we get constraints:
$$u_0(x)=e_2=e_1=e_0=f_1=0, \beta_1=\delta \gamma_3, f_0=\gamma_2-\gamma_3, \gamma_2=\delta \gamma_3^2, \gamma_1=(\gamma_2-\gamma_3)M.$$
This gives us the family
\begin{equation}
[x,y]=z,\quad [x,z]= z +\delta\gamma_3xy + \delta y^2 + (\gamma_3 x^3 +\delta \gamma_3^2 x^2-\gamma_3 (1-\delta \gamma_3)^2 x+\gamma_0),
\end{equation}
and
\begin{equation} [y,z]= (\gamma_3^2\delta -\gamma_3)z,
\end{equation}
which is the family from item (5) in Theorem \ref{thm:main}.

If $M=0$ then we get $\gamma_3=\delta^{-1}$, $\beta_1=-2$, $f_1=3\delta^{-1}$, $f_0=\gamma_2-\delta^{-1}$, $e_2=-3\gamma_3$, $e_1=-2\gamma_2-\delta^{-1}\beta_1$, $e_0=-\gamma_1$

and our
commutator brackets simplify to 
\begin{equation}
[x,y]=z,\quad [x,z]= z -2xy + \delta y^2 + (\delta^{-1} x^3 +\gamma_2 x^2+\gamma_1 x+\gamma_0),
\end{equation}
and
\begin{equation} [y,z]= u_0(x) - (\gamma_1+(2\gamma_2 -2\delta^{-1})x + 3\delta^{-1}x^2) y +
(\gamma_2-\delta^{-1} +3\delta^{-1} x) z,
\end{equation}
which is family (3) from Theorem~\ref{thm:main} with
\[
\beta=\delta^{-1}.
\]

\subsection{Subcase 4: $\delta\neq 0$, $\alpha=0$, $\beta_1
\neq 0$}
The coefficient $\epsilon$ has already been eliminated before the
case division, and $\alpha=0$ refers to the resulting coefficient.
Since $\beta_1
\neq0$, the translation
\[
x'=x+\beta_1^{-1}\beta_0,\qquad y'=y,\qquad z'=z
\]
allows us to assume that $\beta_0=0$. This translation preserves
$\alpha=0$ and $\beta_1
\neq0$.
 Thus we may assume that our commutators are given by
$$[x,y]=z,\quad [x,z]= \beta_1 x y + (\gamma_3 x^3+ \gamma_2 x^2 + \gamma_1 x + \gamma_0 ) +\delta y^2$$
and
$$[y,z]= u_0(x) + u_1(x) y + u_2(x) z.$$
The Jacobi identity gives $[[x,z],y]=-[x,[y,z]]$, and using Equations
(\ref{eq:xxy}) and (\ref{eq:xxxy}) we find
$$[[x,z],y] = \beta_1 zy + \gamma_3 (3x^2 z-3x[x,z]+[x,[x,z]]) + \gamma_2 (2xz -[x,z]) + \gamma_1 z,$$
while
$$-[x,[y,z]]= -u_1(x) z - u_2(x) \big(\beta_1 xy +(\gamma_3 x^3+ \gamma_2 x^2 + \gamma_1 x + \gamma_0 ) +\delta y^2\big).$$
Comparing coefficients we obtain
$$\beta_1 =- 2\delta\gamma_3, \quad f_1 = 3\gamma_3, \quad  f_0 = \gamma_2,$$
$$e_2 = -3\gamma_3, \quad e_1 = 3\gamma_3 N - 2\gamma_2 - \gamma_3\beta_1, \quad e_0 = \gamma_2 N - \gamma_1,$$
$$N e_2  = N e_1 = N e_0 = N\cdot u_0(x) = 0, \qquad N=-\delta\gamma_3.$$
Since $\beta_1=-2\delta\gamma_3
\neq 0$ we have $\gamma_3
\neq 0$ and $\mathrm{char}(K)
\neq 2$.

In particular, $N=-\delta\gamma_3
\neq 0$, so the last line forces
$$e_2=e_1=e_0=0,\qquad u_0(x)=0.$$
Combining $e_2=0$ with $e_2=-3\gamma_3$ gives $3\gamma_3=0$, and since $\gamma_3
\neq 0$
this is possible only when $\mathrm{char}(K)=3$. Hence there are no deformations in this
subcase unless $K$ has characteristic $3$.

When the characteristic of $K$ is three, the constraints simplify: $e_1=0$ gives
$\gamma_2=\delta\gamma_3^2$; $e_0=0$ gives $\gamma_1=\gamma_2 N=-\delta^2\gamma_3^3$; and
$f_1=3\gamma_3=0$, $f_0=\gamma_2=\delta\gamma_3^2$. Thus $u_0=u_1=0$ and
$u_2=\delta\gamma_3^2$, giving the family
$$[x,y]=z,\quad [x,z]= \delta\gamma_3\, xy + \delta y^2 + \big(\gamma_3 x^3 +\delta \gamma_3^2 x^2-\delta^2\gamma_3^3\, x+\gamma_0\big),$$
$$[y,z]= \delta\gamma_3^2\, z,$$
with $\delta
\neq 0$, $\gamma_3
\neq 0$, $\gamma_0\in K$. This is family (5) of
Theorem~\ref{thm:main} with $\alpha=0$.


\subsection{Subcase 5: $\delta\neq 0$, $\alpha\neq 0, \beta_1=0$}

The coefficient $\epsilon$ has already been eliminated. Since
$\alpha
\neq0$, make the change
\[
x'=\alpha^{-1}x,\qquad y'=\alpha y,\qquad z'=z,
\]
and rename the resulting coefficients. We may therefore assume that
$\alpha=1$. Since this normalization is performed after eliminating
$\epsilon$, it does not disturb the condition $\epsilon=0$. Thus
$$[x,y]=z,\quad [x,z]= z + \beta_0 y + (\gamma_3 x^3+ \gamma_2 x^2 + \gamma_1 x + \gamma_0 ) +\delta y^2,$$
$$ [y,z]= u_0(x) + u_1(x) y + u_2(x) z.$$
The Jacobi identity gives $[[x,z],y]=-[x,[y,z]]$, and using Equations
(\ref{eq:xxy}) and (\ref{eq:xxxy}) the coefficient of $yz$ yields
$$2\delta\gamma_3=0.$$

Suppose first that the characteristic of $K$ is not two.
Then $\gamma_3=0$, and comparing the
remaining coefficients gives $\gamma_2=\gamma_1=0$ together with
$u_0(x)=u_1(x)=u_2(x)=0$. Hence $[y,z]=0$ and
$$[x,y]=z,\quad [x,z]=z+\beta_0 y+\gamma_0+\delta y^2,\quad [y,z]=0,$$

Since $[y,z]=0$, the subalgebra $K[y,z]$ is a commutative polynomial ring
and ${\rm ad}_x$ acts on it as the derivation
\[
\partial(y)=z,
\qquad
\partial(z)=z+\beta_0y+\gamma_0+\delta y^2.
\]
Hence $A=K[y,z][x;\partial]$ is an iterated Ore extension of unipotent
type and lies in family (1).

 We therefore assume for the
remainder of this subcase that the characteristic of $K$ is two.

Under this assumption the constraint $2\delta\gamma_3=0$ is vacuous, and solving the
Jacobi identity $[[x,z],y]=[[y,z],x]$ we obtain the conditions
$$f_0=\gamma_2+\gamma_3,\ f_1=\gamma_3,\ e_2=\gamma_3,\ e_1=\gamma_3+\delta\gamma_3^2,\ e_0=\gamma_1+\gamma_2+\gamma_3+\beta_0\gamma_3+\delta\gamma_3(\gamma_2+\gamma_3)$$
along with $(\delta\gamma_3+1)Z=0$ for all $Z\in\{\gamma_3,\gamma_1+\gamma_2,c_0,c_1,c_2,c_3,c_4\}$.
Thus when $\gamma_3=\delta^{-1}$ our commutators simplify to
$$[x,y]=z,\quad [x,z]=z + \beta_0 y + (\delta^{-1} x^3+ \gamma_2 x^2 + \gamma_1 x + \gamma_0 ) +\delta y^2,$$
and
$$[y,z] = u_0(x) + (\delta^{-1} x^2+\gamma_1+\beta_0\delta^{-1})y + (\delta^{-1} x+\gamma_2+\delta^{-1})z,$$
which is part of family (6) in Theorem \ref{thm:main}.
When $\gamma_3
\neq \delta^{-1}$ we get the additional constraints
$\gamma_3=\gamma_1+\gamma_2=u_0(x)=0$, and so our commutators become
$$[x,y]=z,\quad [x,z]=  z + \beta_0 y + ( \gamma_1 x^2 + \gamma_1 x + \gamma_0 ) +\delta y^2,\quad [y,z]= \gamma_1 z,$$
which is family (7) in Theorem \ref{thm:main}.

\subsection{Subcase 6: $\delta\neq 0$, $\alpha=\beta_1=0$}

The coefficient $\epsilon$ has already been eliminated before the
case division. Thus, with $\alpha=\beta_1=0$, our commutator brackets
have the form

 $$[x,y]=z,\quad [x,z]=   \beta_0 y + (\gamma_3 x^3+ \gamma_2 x^2 + \gamma_1 x + \gamma_0 ) +\delta y^2$$
and
$$[y,z]= u_0(x) + u_1(x) y + u_2(x) z.$$

We compute the coefficient constraints given by the Jacobi identity and obtain the following constraints:
\[
f_0=\gamma_2,\qquad
f_1=3\gamma_3,\qquad
e_2=-3\gamma_3,
\]
\[
e_1=-3\delta\gamma_3^2-2\gamma_2,\qquad
e_0=-\delta\gamma_2\gamma_3-\beta_0\gamma_3-\gamma_1.
\]
Since $\beta_1=0$, Equation~\eqref{eq:IIIb-compact-beta} gives
\[
2\delta\gamma_3=0,
\]
and hence $2\gamma_3=0$.  Moreover, here
$M=-\delta\gamma_3$, and Equation~\eqref{eq:IIIb-compact-M} gives
$Mu_1(x)=0$.  Taking the coefficient of $x^2$ and using
$e_2=-3\gamma_3$ gives
\[
3\delta\gamma_3^2=0,
\]
and hence $3\gamma_3^2=0$.  It follows that $\gamma_3=0$.

Thus our commutators simplify to
\begin{equation}
[x,y]=z,\quad [x,z]=\beta_0 y + \gamma_2 x^2 + \gamma_1 x + \gamma_0 +\delta y^2, \quad [y,z]=u_0(x) + (-2\gamma_2 x-\gamma_1)y +  \gamma_2 z.
\end{equation}
In the case that the characteristic of $K$ is equal to $2$, this simplifies to 
$$[x,y]=z, \quad [x,z]= \beta_0 y +  \gamma_2 x^2 + \gamma_1 x + \gamma_0 +\delta y^2,$$ $[y,z]=u_0(x)+\gamma_1y +\gamma_2 z$, which  
is just family (6) in Theorem \ref{thm:main}.

If the characteristic of $K$ is not $2$ then we can make a substitution $x'=x, y'=y+\delta^{-1}\beta_0/2, z'=z$ and our family simplifies to
\begin{equation}
[x,y]=z,\quad [x,z]= \gamma_2 x^2 + \gamma_1 x + \gamma_0 +\delta y^2, \quad [y,z]=u_0(x) + (-2\gamma_2 x-\gamma_1)y +  \gamma_2 z.
\end{equation}
This is a subfamily of the family (3) in Theorem \ref{thm:main}.

\section{Proof of Theorem \ref{thm:main}}\label{sec:main}

The calculations in Sections~\ref{sec:I}--\ref{sec:IIIb} show that
every filtered deformation under consideration is isomorphic to one
of the algebras in Theorem~\ref{thm:main}. It remains to explain
carefully why the displayed presentations have the required PBW
property and why the indicated filtrations have polynomial associated
graded rings.

We first clarify the role of family (1). The algebras placed in
family (1) are the Ore-extension presentations obtained in Cases I
and II and in the Ore-extension branches of Cases III(a) and III(b).
For these algebras no separate verification is required. Each was
obtained from an algebra $A$ that was assumed at the outset to be a
filtered deformation of $K[x,y,z]$, and the Ore generators are the
original generators of $A$ after the invertible substitutions made in
the relevant case. Each such presentation therefore carries a compatible filtration, after the reweightings described above, whose associated graded ring is a polynomial ring.

We stress that no assertion is being made here that an arbitrary
iterated Ore extension of unipotent type is a filtered deformation of
a polynomial ring. Family (1) consists only of the presentations
arising in Sections~\ref{sec:I}--\ref{sec:IIIb}, each equipped with
the filtration it inherits there.

We now verify the presentations in families (2)--(10). Give the
free monoid on $x,y,z$ the weighted degree-lexicographic order:
words are first compared by weighted degree and words of equal
weighted degree are compared lexicographically using
\[
z>y>x.
\]
The following weights may be used. Whenever the degree of a zero
polynomial occurs in the table, we take that degree to be zero.

\begin{equation}
\label{eq:PBW-weight-table}
\begin{array}{c|c|c}
\text{family}
&
(\deg x,\deg y,\deg z)
&
\text{choice of }b
\\ \hline
(2)
&
(2,b,b+1)
&
b\geq\max\{2,\deg f\}
\\
(3)
&
(6,10,15)
&
\text{---}
\\
(4)
&
(2,b,b+1)
&
b\geq\max\{2,\deg u\}
\\
(5),(6),(7)
&
(6,10,15)
&
\text{---}
\\
(8)
&
(2,b,b+1)
&
b\geq\max\{4,\deg r\}
\\
(9)
&
(2,b,b+1)
&
b\geq\max\{2,2\deg q,\deg r\}
\\
(10)
&
(2,b,b+1)
&
b\geq\max\{2,\deg u,2\deg v+1\}.
\end{array}
\end{equation}
For each family, write the relations in the form
\[
[x,y]=R_{xy},\qquad
[x,z]=R_{xz},\qquad
[y,z]=R_{yz}.
\]
We orient them as the rewriting rules
\begin{equation}
\label{eq:PBW-rewriting-rules}
yx\longrightarrow xy-R_{xy},
\qquad
zx\longrightarrow xz-R_{xz},
\qquad
zy\longrightarrow yz-R_{yz}.
\end{equation}
The ordered quadratic word on the right of each rule has the same
weight as the word on the left but is lexicographically smaller.
Every monomial occurring in $R_{xy}$, $R_{xz}$, or $R_{yz}$ has
strictly smaller weight than the corresponding quadratic word.

We record the weight checks. For families (3) and (5)--(7), the
weights of $xy,xz,yz$ are respectively
\[
16,\qquad21,\qquad25.
\]
The terms in $R_{xy}$ have weight at most $15$, those in $R_{xz}$
have weight at most $20$, and those in $R_{yz}$ have weight at most
$24$. Here the last bound uses $\deg u\leq4$ in families (3) and
(6).

For weights $(2,b,b+1)$, the weights of $xy,xz,yz$ are respectively
\[
b+2,\qquad b+3,\qquad2b+1.
\]
The inequalities in the table immediately give the required strict
bounds for families (2) and (4). For family (8), the potentially
largest additional terms have weights
\[
2\deg r,\qquad b+4,\qquad b+3,\qquad2b,
\]
corresponding respectively to $r(x)$, $x^2y$, $xz$, and $y^2$.
All are strictly smaller than $2b+1$ when
$b\geq\max\{4,\deg r\}$.

For family (9), put $d=\deg q$. The Hasse--Schmidt coefficients
satisfy
\[
\deg q^{[n]}\leq d-n
\]
whenever $q^{[n]}
\neq0$. Hence
\[
\deg q^{[1]}\leq d-1,
\qquad
\deg H_\epsilon(q)\leq d-2
\]
when the respective polynomials are nonzero. The choice
\[
b\geq\max\{2,2d,\deg r\}
\]
therefore makes the terms
\[
q(x),\qquad
q^{[1]}(x)y,\qquad
H_\epsilon(q)(x)z,\qquad
r(x),\qquad
\epsilon y^2
\]
strictly smaller than the corresponding left-hand sides.

Finally, for family (10), the choice
\[
b\geq\max\{2,\deg u,2\deg v+1\}
\]
gives
\[
2\deg u<2b+1,
\]
\[
2\deg v+b<2b+1,
\qquad
2\deg v+b+1<2b+1.
\]
Thus the terms in
\[
u(x)+v(x)(\epsilon y+z)+\epsilon y^2
\]
all have weight strictly smaller than $yz$.

It follows that every rule in
Equation~\eqref{eq:PBW-rewriting-rules} strictly decreases the
weighted degree-lexicographic order. The irreducible words are
precisely
\[
x^iy^jz^k,
\qquad i,j,k\geq0.
\]

There are no inclusion ambiguities, and the only overlap ambiguity
is the word
\[
zyx.
\]
Reducing first $zy$ and first $yx$, respectively, gives expressions
whose difference is
\[
[z,R_{xy}]+[R_{xz},y]+[x,R_{yz}].
\]
Since
\[
R_{xy}=[x,y],\qquad
R_{xz}=[x,z],\qquad
R_{yz}=[y,z],
\]
this difference is the Jacobi expression
\[
[z,[x,y]]+[[x,z],y]+[x,[y,z]],
\]
whose reduction to normal form is zero by the coefficient calculations in Sections \ref{sec:IIIa} and \ref{sec:IIIb}. 
The Jacobi identities for the displayed families were
established in the calculations of
Sections~\ref{sec:IIIa} and~\ref{sec:IIIb}; equivalently, they may be
verified directly by substitution into the displayed relations.
Thus the overlap ambiguity is resolvable.

Bergman's diamond lemma now shows that
\[
x^iy^jz^k,
\qquad i,j,k\geq0,
\]
is a $K$-basis for each algebra in families (2)--(10).

For the weights assigned in
Equation~\eqref{eq:PBW-weight-table}, define
\[
V_N
:=
\operatorname{span}_K
\left\{
x^iy^jz^k:
i\deg x+j\deg y+k\deg z\leq N
\right\}.
\]
Because all weights are positive, each $V_N$ is finite-dimensional
and
\[
\bigcup_{N\geq0}V_N=A.
\]
The decreasing rewriting rules show that
\[
V_MV_N\subseteq V_{M+N}.
\]
In the associated graded algebra, all three commutators vanish, so
there is a surjective graded homomorphism
\[
K[\bar x,\bar y,\bar z]
\longrightarrow
\operatorname{gr}A.
\]
The PBW basis shows that this homomorphism is injective. Hence
\[
\operatorname{gr}A
\cong
K[\bar x,\bar y,\bar z]
\]
with the weighted grading. Therefore every displayed algebra in
families (2)--(10) is a filtered deformation of a three-variable
polynomial ring.

Together with the filtration-compatible Ore-extension cases and
the exhaustive calculations in
Sections~\ref{sec:I}--\ref{sec:IIIb}, this proves
Theorem~\ref{thm:main}.

\section{Etingof's question for three-variable polynomial rings}\label{sec:et}
We can now show that in positive characteristic a filtered deformation of a polynomial ring in three variables is PI.  To do this, we use the classification provided in Theorem \ref{thm:main}. We need a criterion that ensures certain filtered deformations are PI.  To obtain this we first prove a basic lemma.

\begin{lemma}
Let $F$ be an algebraically closed field, let $A$ be an $F$-algebra
having an ascending filtration by finite-dimensional $F$-vector spaces
whose associated graded ring is a finitely generated commutative
domain, and let $D$ be the quotient division algebra of $A$.  Then
every subfield $K$ of $D$ containing $F$ is finitely generated as an
extension of $F$.
\label{lem:vamos}
\end{lemma}

\begin{proof}
Let $B=\operatorname{gr}A$ and let $K$ be a subfield of $D$ containing
$F$.  Since $B$ is a finitely generated commutative $F$-algebra,
\[
B\otimes_F K
\]
is a noetherian $K$-algebra.  This is the associated graded ring of
$A\otimes_F K$ with respect to the induced filtration, and hence
$A\otimes_F K$ is noetherian.

Let
\[
S=A\setminus\{0\}.
\]
Since $A$ is a noetherian domain, $S$ is a left and right Ore set and
\[
D=S^{-1}A=AS^{-1}.
\]
Localizing after extension of scalars gives
\[
D\otimes_F K
\cong
(S\otimes 1)^{-1}(A\otimes_F K),
\]
and hence $D\otimes_F K$ is noetherian.

Now view $K\otimes_F K$ as a subring of $D\otimes_F K$.  Choosing a
left $K$-basis of $D$ shows that $D\otimes_F K$ is a free, and hence
faithfully flat, left $K\otimes_F K$-module; choosing a right
$K$-basis gives the analogous statement on the right.  It follows by
faithfully flat descent that $K\otimes_F K$ is noetherian.  By
\cite{Vam}, $K$ is finitely generated over $F$.
\end{proof}

\begin{prop}
\label{prop:invoke}
Let $F$ be an algebraically closed field of positive characteristic $p>0$ and let $A$ be a filtered deformation of a three-variable polynomial ring over $F$.  If $A$ has a central element $\Theta \in A\setminus F$, then $A$ is PI.
\end{prop}

\begin{proof}
This proof follows the strategy from \cite{Bell}.  
Put $B=\operatorname{gr}A$ and $h=\bar\Theta$.  Since $h$ is a
nonconstant homogeneous element of the polynomial ring $B$, we can choose
homogeneous $f,g\in B$ such that $B$ is finite over
$S=F[h,f,g]$.  Since $B$ is Cohen--Macaulay, it is in fact finite
free over $S$.  Choose a homogeneous $S$-basis
$\bar w_1=1,\ldots,\bar w_m$, and filtered lifts $Y,Z,w_1,\ldots,w_m$ of $f,g,\bar w_1,\ldots, \bar w_m$ respectively.
Then
\[
 \Theta^iY^jZ^kw_\ell
 \qquad(i,j,k\geq0,\ 1\leq\ell\leq m)
\]
is a filtration-compatible $F$-basis of $A$.

Since $h$ and $f$ are algebraically independent, their filtered lifts
$\Theta$ and $Y$ are algebraically independent.  Hence
\[
F[\Theta,Y]\cong F[s,t]
\]
and
\[
K:=F(\Theta,Y)
\]
has transcendence degree two over $F$.

Viewing $K$ as a subfield of $\operatorname{Frac}(A)$, we have
\[
 KV_n=
 \bigoplus_{kc+\deg(w_\ell)\leq n}KZ^kw_\ell,
 \qquad c=\deg Z.
\]
Indeed, this follows from the filtration-compatible basis above by
clearing left denominators in $K=F(\Theta,Y)$.

Consequently, for every fixed $\rho$,
\[
 \dim_K(KV_n/KV_{n-\rho})=O(1),
\]
where the implied bound is independent of $n$.

Put $e=\deg Y$ and
\[
 \Omega_n=\operatorname{span}_K\{u\in A:[Y,u]\in V_n\}.
\]
Choose a power $q$ of $p$ with $q-1>c$, and then $L$ such that
\[
 e<\left(1-\frac{c}{q-1}\right)q^L.
\]
For $N>L$ and $L\leq i\leq N$, set
\[
 a_i=q^N-\sum_{j=i}^{N-1}q^j,\qquad
 b_N=c\left(q^N-\sum_{j=L-1}^{N-1}q^j\right).
\]
Then
\[
a_{i+1}-a_i=q^i>0,
\]
so the $a_i$ are distinct; moreover, $q^i\mid a_i$, and

\[
 0<ca_i-b_N<\frac{c}{q-1}q^i.
\]
Since $(\operatorname{ad}_v)^{p^r}
=\operatorname{ad}_{v^{p^r}}$ in characteristic $p$,
\[
 \deg [Y,Z^{a_i}]
 \le ca_i-q^i+e\le b_N.
\]

Hence $Z^{a_i}\in\Omega_{b_N}$, while
$ca_i>b_N$, so the displayed description of $KV_{b_N}$ gives
\[
Z^{a_i}\notin KV_{b_N}.
\]
Also $KV_{b_N-e+1}\subseteq\Omega_{b_N}$.
The adapted basis therefore gives
\[
 KZ^{a_L}\oplus\cdots\oplus KZ^{a_N}
 \oplus KV_{b_N-e+1}\subseteq\Omega_{b_N}.
\]
We let $E$ denote the centralizer of $Y$ in the quotient division algebra of $A$. Then rank--nullity applied to
$\operatorname{ad}_Y$ on the space above gives
\[
 \dim_KE\geq
 N-L+1-\dim_K(KV_{b_N}/KV_{b_N-e+1}).
\]
The last term is bounded independently of $N$, and so
$[E:K]=\infty$.

Now $K\subseteq Z(E)$.  Choose a maximal subfield $K'$ of $E$
containing $K$.  Then $[K':K]=\infty$;
otherwise \cite[Theorem 15.4]{Lam} would imply
$[E:K]<\infty$.  By Lemma~\ref{lem:vamos}, $K'$ is finitely generated
over $F$, so
\[
 \operatorname{trdeg}_F K'>\operatorname{trdeg}_F K=2.
\]
Thus $\operatorname{trdeg}_F K'\geq3$.  Since $A$ is an affine
Ore domain of Gelfand--Kirillov dimension three,
\cite[Theorem 1.1]{Belldiv} applies and implies that $A$ is PI.
\end{proof}

We are now ready to prove Theorem \ref{thm:main2}. To do this, we show that the algebras we consider have a non-trivial centre and hence we can invoke Proposition \ref{prop:invoke}. In some families in Theorem \ref{thm:main} (particularly families (3) and (4)), we find centres in part via computer search using Sage Math (code can be found in the appendix of the Master's thesis of the second-named author \cite{Li}).  
We prove the existence of non-trivial central elements for these algebras in the following lemmas.  We first need a lemma that allows us to find a non-trivial centre in characteristic zero to deduce the existence of a non-trivial centre in positive characteristic.
\begin{lemma}\label{lem:lift}
Let $R$ be a finitely generated commutative $\mathbb{Z}$-algebra which is an integral domain,
let $F=\operatorname{Frac}(R)$ be its field of fractions, and let $\theta: R\to K$ be a ring homomorphism from $R$ to a field $K$, so that we view $K$ as an $R$-module.  Let $A$ be an associative
$R$-algebra equipped with an ascending filtration
$$R=A_0\subseteq A_1\subseteq A_2\subseteq\cdots,\qquad A=\bigcup_{N\ge 0}A_N,$$
such that
\begin{enumerate}
\item each $A_N$ is a finitely generated free $R$-module and each quotient
      $A_N/A_{N-1}$ is a free $R$-module;
\item $A_M A_N\subseteq A_{M+N}$ for all $M,N\ge 0$; and
\item $A$ is generated as an $R$-algebra by finitely many elements
      $g_1,\dots,g_r$, all lying in $A_d$ for some fixed $d$.
\end{enumerate}
If $A\otimes_R F$ has  a central element not contained in $F\cdot 1$, then $A\otimes_R K$ has  a central element not contained in $K\cdot 1$.
\end{lemma}

\begin{proof}
Write
\[
A_F:=A\otimes_R F.
\]
Since every $A_N$ is a free $R$-module, $A$ is torsion-free over
$R$, and hence the natural map
\[
A\longrightarrow A_F
\]
is injective.

Choose
\[
\xi\in Z(A_F)\setminus F\cdot1.
\]
Since $A_F=\bigcup_N(A_N\otimes_R F)$, there is some $N$ such that
$\xi\in A_N\otimes_R F$. Clearing the finitely many denominators of
the coordinates of $\xi$ in an $R$-basis of $A_N$, choose a nonzero
$r\in R$ such that
\[
a:=r\xi\in A_N.
\]
Multiplication by the nonzero scalar $r$ does not make a nonscalar
element scalar, so
\[
a
\notin F\cdot1.
\]
Moreover, $a$ is central in $A_F$. For every generator $g_i$,
the commutator $[a,g_i]$ lies in $A$ and maps to zero in $A_F$.
Injectivity of $A\to A_F$ therefore gives
\[
[a,g_i]=0.
\]
Thus $a\in Z(A)\cap A_N$.

Since the $g_i$ generate $A$ as an $R$-algebra, an element is
central if and only if it commutes with every $g_i$. Define the
$R$-linear map
\[
\Phi\colon A_N\longrightarrow (A_{N+d})^r,
\qquad
b\longmapsto
\bigl([b,g_1],\ldots,[b,g_r]\bigr).
\]
Then
\[
\ker\Phi=Z(A)\cap A_N.
\]

Choose $R$-bases of $A_N$ and $(A_{N+d})^r$, of respective sizes
$q$ and $p$, and let $Y\in M_{p\times q}(R)$ be the matrix of
$\Phi$. Viewed over $F$, the vectors $1$ and $a$ are linearly
independent elements of $\ker\Phi$. Consequently,
\[
\operatorname{rank}_F(Y)\leq q-2.
\]
It follows that every $(q-1)\times(q-1)$ minor of $Y$ vanishes in
$R$.

Apply the homomorphism $\theta\colon R\to K$ entrywise to $Y$.
All $(q-1)\times(q-1)$ minors of $\theta(Y)$ vanish, and hence
\[
\operatorname{rank}_K\theta(Y)\leq q-2.
\]
Therefore,
\[
\dim_K\ker(\Phi\otimes_R K)\geq2.
\]
Since each quotient $A_{j}/A_{j-1}$ is free over $R$, the exact
sequence
\[
0\longrightarrow A_{j-1}\longrightarrow A_j
\longrightarrow A_j/A_{j-1}\longrightarrow0
\]
remains exact after tensoring with $K$.  Hence the maps
\[
A_j\otimes_R K\longrightarrow A_{j+1}\otimes_R K
\]
are injective, and we may regard $A_N\otimes_R K$ as a subspace of
$A\otimes_R K$.

The kernel is precisely the subspace of $A_N\otimes_R K$ consisting
of elements that commute with the images of $g_1,\ldots,g_r$.
Thus
\[
\ker(\Phi\otimes_R K)
=
Z(A\otimes_R K)\cap(A_N\otimes_R K).
\]
Since the scalar subspace has dimension one, this kernel contains a
nonscalar central element.
\end{proof}

\begin{lemma}\label{lem:center3}
Let $K$ be a field and let $A$ belong to family {\rm (3)} of
Theorem~\ref{thm:main}.  Then $Z(A)\neq K$.
\end{lemma}

\begin{proof}
Let
\[
R=\mathbb Z[\delta,\beta,\gamma_0,\gamma_1,\gamma_2,c_0,\ldots,c_4],
\qquad F=\operatorname{Frac}(R),
\]
and let $\mathcal A$ be the universal algebra given by the relations
of family {\rm (3)}, with
\[
u(x)=\sum_{i=0}^4c_ix^i.
\]

Give $x,y,z$ weights $6,10,15$, respectively.  With weighted
degree-lexicographic order and $z>y>x$, the relations give decreasing
rewriting rules for $yx,zx,zy$.  The only overlap ambiguity is $zyx$.
Reducing this overlap in the two possible ways and putting the results
in normal form, their difference is zero identically over $R$.
Hence Bergman's diamond lemma gives the PBW basis

\[
x^iy^jz^k\qquad(i,j,k\geq0).
\]
In particular, the corresponding filtration satisfies the hypotheses
of Lemma~\ref{lem:lift}.

A symbolic normal-form computation gives the following central element
of $\mathcal A$:
\[
C=C_6+C_5+\cdots+C_1,
\]
where
$$C_6=20\delta y^3-30z^2, \quad C_5= -12c_4x^5-20\Lambda x^3y+30\Lambda x^2z-20\delta \Lambda xy^2+10\delta \Lambda yz,$$
\begin{align*} C_4 &=-\big(8\delta^2c_4^2-3\beta\delta c_4+15c_3\big)x^4-10\big(18\beta^2\delta-11\beta\delta^2c_4+3\delta c_3-6\gamma_2\big)x^2y\\
&+30\big(5\beta^2\delta-3\beta\delta^2c_4+\delta c_3-2\gamma_2\big)xz-10\delta\big(4\beta^2\delta-\beta\delta^2c_4+\delta c_3-4\gamma_2\big)y^2,\end{align*}
$$C_3=-\big(8\delta^2c_3c_4-21\beta^2\delta^2c_4-35\beta\delta c_3+56\delta\gamma_2c_4+20c_2\big)x^3-20\Xi xy+10\Xi z,$$
\begin{align*} C_2 &=-\big(8\delta^2c_2c_4-3\beta^3\delta^3c_4-5\beta^2\delta^2c_3-7\beta\delta^2\gamma_2c_4-40\beta\delta c_2\\
&+54\delta\gamma_1c_4+15\delta\gamma_2c_3+30c_1\big)x^2 +10\big(\delta^2\gamma_1c_4-\delta c_1+6\gamma_0\big)y\end{align*}
$$C_1=-\big(8\delta^2c_1c_4-9\beta\delta^2\gamma_1c_4-30\beta\delta c_1+48\delta\gamma_0c_4+15\delta\gamma_1c_3+60c_0\big)x$$

where $\Lambda:=2\delta c_4-3\beta$ and $$\Xi :=-3\beta^3\delta^2+2\beta^2\delta^3c_4-\beta\delta^2c_3+3\beta\delta\gamma_2-\delta^2\gamma_2c_4+\delta c_2-3\gamma_1.$$

A direct normal-form calculation gives
\[
[C,x]=[C,y]=[C,z]=0.
\]
Moreover, the coefficient of $z^2$ in $C$ is $-30$, so $C$ is
nonscalar in $\mathcal A\otimes_R F$. 
We do not claim that this particular element remains nonscalar after
specialization; indeed, its $z^2$-coefficient may vanish in positive
characteristic.
 Every algebra in family {\rm
(3)} is obtained by specialization $R\to K$, and
Lemma~\ref{lem:lift} therefore gives a nonscalar central element after
every such specialization.
\end{proof}

\begin{lemma}\label{lem:center4}
Let $K$ be a field and let $A$ belong to family {\rm (4)} of
Theorem~\ref{thm:main}.  Then $Z(A)\neq K$.
\end{lemma}

\begin{proof}
Write
\[
[x,y]=z,\qquad [x,z]=z-2xy+\zeta,\qquad [y,z]=y^2+u(x),
\]
and let $n=\deg u$ (with $n=0$ if $u=0$).  Put
\[
R=\mathbb Z[t_0,\ldots,t_n,\eta],\qquad F=\operatorname{Frac}(R),
\]
and let $\mathcal A$ be the universal $R$-algebra with relations
\[
[x,y]=z,\qquad [x,z]=z-2xy+\eta,\qquad
[y,z]=y^2+t_0+t_1x+\cdots+t_nx^n.
\]
The given algebra $A$ is obtained from $\mathcal A$ by specializing
$\eta\mapsto\zeta$ and the $t_i$ to the coefficients of $u(x)$.

Give $x,y,z$ weights
\[
\deg x=2,\qquad \deg y=b,\qquad \deg z=b+1,
\qquad b\geq\max\{2,n\}.
\]

With weighted degree-lexicographic order and $z>y>x$, the relations
give decreasing rewriting rules for $yx,zx,zy$.  The only overlap
ambiguity is $zyx$, and direct reduction in the two possible ways
gives the same normal form identically over $R$.  Hence Bergman's
diamond lemma gives the PBW basis
\[
x^iy^jz^k\qquad(i,j,k\geq0).
\]
Thus Lemma~\ref{lem:lift} applies.

We work over $F$ and write
\[
u(x)=t_0+\cdots+t_nx^n.
\]
A direct calculation shows that
\[
\widetilde C=z^2+2xy^2-2yz+2y^2-2\eta y
\]
commutes with $x$, and that
\[
[\widetilde C,y]=[-u(x),y]-u(x)z-zu(x).
\]
Thus it suffices to show that $u(x)z+zu(x)=[g(x),y]$ for some
$g(x)\in F[x]$.
Since $z=[x,y]$,
\[
u(x)z+zu(x)
=[u(x)x,y]+xyu(x)-u(x)yx.
\]

By linearity it therefore suffices to prove that, for every $n\geq0$,
there is $q_n(x)\in F[x]$ such that
\[
[q_n(x),y]=xyx^n-x^nyx.
\]
For $n=0$ we may simply take $q_0(x)=x$.  We therefore assume
$n\geq1$ for the remainder of the argument.

Let $s,t$ act by left and right multiplication by $x$, respectively,
and let
\[
I=\{P(s,t)\in F[s,t]:P(s,t)y=0\}.
\]
From
\[
[x,[x,y]]=[x,y]-2xy+\eta
\]
we obtain
\[
(s-t)\big((s-t)^2+s+t\big)\in I.
\]
Put
\[
J=((s-t)^2+s+t).
\]
Since
\[
q(s)-q(t)-st^n+s^nt
\]
is divisible by $s-t$, it is enough to choose $q$ so that this
expression lies in $J$.  Indeed, $F[s,t]$ is a UFD and $s-t$ has no
common irreducible factor with
\[
(s-t)^2+s+t.
\]
Hence any polynomial divisible by $s-t$ and by the generator of $J$
is divisible by their product, which lies in $I$.

Now
\[
F[s,t]/J\cong F[T],
\qquad
s=\frac{T-T^2}{2},\quad
t=\frac{-T-T^2}{2}.
\]
Interchanging $s$ and $t$ corresponds to $T\mapsto-T$.  The
polynomials $q(s)-q(t)$ map precisely onto the odd polynomials in
$T$: indeed, $s^m-t^m$ has degree $2m-1$ with nonzero leading
coefficient.  Since $st^n-s^nt$ is antisymmetric in $s,t$, its image
is odd, and the required $q_n$ exists.

Hence there is $f(x)\in F[x]$ such that
\[
\widetilde C+f(x)
\]
commutes with both $x$ and $y$, and therefore is central.  It is
nonscalar since its coefficient of $z^2$ is $1$.  Lemma~\ref{lem:lift}
now gives a nonscalar central element after every specialization
$R\to K$.
\end{proof}

\begin{proof}[Proof of Theorem \ref{thm:main2}]

The families of type (1) are iterated Ore extensions.  A unipotent automorphism has order a power of $p$ when the base field has characteristic $p>0$ and so the fact that these algebras are PI follows from work of Brown and Zhang \cite[Theorem 3.3]{BZ}. 

For the remaining families, in light of
Proposition~\ref{prop:invoke}, it suffices to show that every family
from (2)--(10) has a non-trivial centre in positive characteristic.

For the algebras from family (2) with generators $x,y,z$ and relations $[x,y]=z, \qquad [x,z]=\gamma y, \qquad [y,z]=f(x)$ for some $\gamma\in K$ and some $f(x)\in K[x]$, we have
${\rm ad}_x^n(y) = \gamma^{(n-1)/2} z$, and ${\rm ad}_x^n(z)=\gamma^{(n+1)/2}y$ if $n$ is odd; and ${\rm ad}_x^n(y)=\gamma^{n/2} y, {\rm ad}_x^n(z)=\gamma^{n/2} z$ if $n$ is even.
In particular, since ${\rm ad}_x^{p^i} = {\rm ad}_{x^{p^i}}$ we see that for all fields of characteristic $p>0$, $x^{p^2}-\gamma^{(p^2-p)/2} x^p$ is central.  Thus we obtain the desired result for the algebras in (2).

Families (3) and (4) have non-trivial centre by
Lemmas~\ref{lem:center3} and \ref{lem:center4}
and thus these families satisfy polynomial identities in positive characteristic.

Families (5) and (7) have relations of the form
$[x,y]=z$ and $[y,z]=cz$ for some $c\in K$. Thus
\[
\operatorname{ad}_y^n(x)=-c^{n-1}z
\qquad(n\geq1).
\]
Since
\[
\operatorname{ad}_{y^p}
=
\operatorname{ad}_y^p,
\]
the element
\[
y^p-c^{p-1}y
\]
is central.

For family (6) we have characteristic $2$. When $\alpha=0$ we have relations
$$[x,y]=z, \quad [x,z]= \beta y +  \gamma_2 x^2 + \gamma_1 x + \gamma_0 +\delta y^2,\quad [y,z]=u(x)+\gamma_1y +\gamma_2 z,$$

It follows that ${\rm ad}_x$ preserves the $8$-dimensional $K$-vector space $V$ spanned by $1, x, x^2, x^3, x^4, y, z, y^2$.  Since the space of $K$-linear endomorphisms of $V$ is $64$-dimensional, and since ${\rm ad}_x^{p^i}={\rm ad}_{x^{p^i}}$ for $i\ge 1$, it follows that there is a non-trivial $K$-linear dependence between 
${\rm ad}_x, {\rm ad}_{x^p}, \ldots , {\rm ad}_{x^{p^{64}}}$ when we restrict the action to $V$.  Consequently there is a non-trivial polynomial in $x$ that commutes with $y$ and $z$ and hence is central. When $\alpha=1$, we have the relations
$$[x,y]=z, [x,z]=z + \beta y + (\delta^{-1} x^3+ \gamma_2 x^2 + \gamma_1 x + \gamma_0 ) +\delta y^2,$$
and
$$[y,z] = u(x) + (\delta^{-1} x^2+\gamma_1+\beta \delta^{-1})y + (\delta^{-1} x+\gamma_2+\delta^{-1})z.$$
Then we see that the $K(x)$-vector subspace of the quotient division algebra of our ring that is spanned by $1, y, z, y^2$ is invariant under ${\rm ad}_x$ and moreover, we have
$${\rm ad}_x(1)=0, {\rm ad}_x(y)=z, {\rm ad}_x(z)= z + \beta y + (\delta^{-1} x^3+ \gamma_2 x^2 + \gamma_1 x + \gamma_0 ) +\delta y^2$$
and $${\rm ad}_x(y^2) = {\rm ad}_y(z)=u(x) + (\delta^{-1} x^2+\gamma_1+\beta \delta^{-1})y + (\delta^{-1} x+\gamma_2+\delta^{-1})z.$$
Then if we use this fact we can verify that $$h(x)=x^4 + x^3 + (1+\beta+\delta\gamma_2)x^2 + (\beta+\delta\gamma_1)x$$ is central and so we are done.

Finally, consider families (8), (9), and (10), all of which occur in
characteristic two. For family (8), let
\[
W=\operatorname{span}_K\{1,x,x^2,x^3,y,z\}.
\]
The defining relations show that
\[
\operatorname{ad}_x(W)\subseteq W.
\]

For family (9), let $d=\deg q$ when $q
\neq 0$, and take
\[
W=\operatorname{span}_K\{1,x,\ldots,x^d,y,z\};
\]
when $q=0$, take $W=\operatorname{span}_K\{1,y,z\}$. Since
\[
[x,y]=z,\qquad [x,z]=\epsilon z+q(x),
\]
this space is again invariant under $\operatorname{ad}_x$.

For family (10), the space
\[
W=\operatorname{span}_K\{1,y,z\}
\]
is invariant under $\operatorname{ad}_x$, since
\[
[x,y]=z,\qquad [x,z]=\epsilon z.
\]

In each of these cases, put $m=\dim_K W$. Since
$\operatorname{End}_K(W)$ has dimension $m^2$, the endomorphisms
\[
\operatorname{ad}_x,\quad
\operatorname{ad}_x^{p},\quad
\operatorname{ad}_x^{p^2},\quad
\ldots,\quad
\operatorname{ad}_x^{p^{m^2}}
\]
are linearly dependent on $W$. Choose scalars
$\lambda_0,\ldots,\lambda_{m^2}\in K$, not all zero, such that
\[
\sum_{i=0}^{m^2}
\lambda_i\operatorname{ad}_x^{p^i}
\]
vanishes on $W$. In characteristic $p$,
\[
\operatorname{ad}_x^{p^i}
=
\operatorname{ad}_{x^{p^i}},
\]
and hence the nonzero polynomial
\[
P(x):=\sum_{i=0}^{m^2}\lambda_i x^{p^i}
\]
commutes with every element of $W$. In particular, it commutes with
$y$ and $z$, and it plainly commutes with $x$. Since $x,y,z$
generate the algebra, $P(x)$ is central. The PBW basis shows that
$P(x)$ is nonconstant and nonzero. Thus families (8), (9), and (10)
all have nontrivial centre.
\end{proof}
\section*{Acknowledgment}
We thank the referee for many helpful comments and suggestions and in particular for suggesting repairs to some of the earlier proofs. 
\section*{AI Declaration}
What a bizarre world we live in! The computer code used for the symbolic calculations in this paper
was written by the second-named author without generative-AI
assistance. During the preparation and revision of the manuscript,
the authors used generative-AI tools to assist with proofreading,
checking the consistency and completeness of the case analysis,
analyzing and organizing the referee's comments, and drafting
possible revisions to the exposition and mathematical arguments.


 \end{document}